\RequirePackage{fix-cm}
\documentclass{svjour3}                     
\smartqed  
\usepackage{graphicx}
\usepackage{cite}

\usepackage{mathptmx}      
\usepackage{latexsym}
\usepackage{subfigure}
\usepackage{amsmath}
\usepackage{amssymb }
\usepackage{booktabs}

\usepackage{multicol}
\usepackage{multirow}
 \usepackage{array}
\usepackage{mathtools}
\usepackage{amsfonts}
\usepackage[misc,geometry]{ifsym} 

\usepackage{color}

\newtheorem{algorithm}{Algorithm}
\usepackage{graphicx}
\graphicspath{ {./images/} }
\begin{document}
	
	\title{Self-adaptive algorithms for quasiconvex programming and applications to machine learning
	}
	
	
	\author{Tran Ngoc Thang \and Trinh Ngoc Hai
	}
	
	
	\institute{
			Tran Ngoc Thang \at
		School of Applied Mathematics and Informatics, Hanoi University of Science and Technology,\\
		1st Dai Co Viet, Hai Ba Trung, Hanoi, Viet Nam \\
		\email{thang.tranngoc@hust.edu.vn}
		\and
		Trinh Ngoc Hai (\;\Letter\;)  \at 
	    School of Applied Mathematics and Informatics,
	    Hanoi University of Science and Technology,
	    \\1st Dai Co Viet, Hai Ba Trung, Hanoi, Viet Nam\\
		\email{hai.trinhngoc@hust.edu.vn}
	}
	
	\date{\today}

	\maketitle
	
	\begin{abstract}
For solving a broad class of nonconvex programming problems on an unbounded constraint set, we provide a self-adaptive step-size strategy that does not include line-search techniques and establishes the convergence of a generic approach under mild assumptions. Specifically, the objective function may not satisfy the convexity condition. Unlike descent line-search algorithms, it does not need a known Lipschitz constant to figure out how big the first step should be. The crucial feature of this process is the steady reduction of the step size until a certain condition is fulfilled.
In particular, it can provide a new gradient projection approach to optimization problems with an unbounded constrained set. The correctness of the proposed method is verified by preliminary results from some computational examples. To demonstrate the effectiveness of the proposed technique for large-scale problems, we apply it to some experiments on machine learning, such as supervised feature selection, multi-variable logistic regressions and neural networks for classification. 

		\keywords{nonconvex programming \and gradient descent algorithms \and quasiconvex functions
		\and pseudoconvex functions
		\and self-adaptive step-sizes	}
		
		\subclass{90C25
		\and 90C26  \and 68Q32 \and 93E35}
	\end{abstract}
	\section{Introduction}\label{s1}
	
Gradient descent methods are a common tool for a wide range of programming problems, from convex to nonconvex, and have numerous practical applications (see \cite{Boyd}, \cite{Cevher}, \cite{Lan} and references therein). At each iteration, gradient descent algorithms provide an iterative series of solutions based on gradient directions and step sizes.
For a long time, researchers have focused on finding the direction to improve the convergence rate of techniques, while the step-size was determined using one of the few well-known approaches (see \cite{Boyd}, \cite{Nesterov}). 

Recently, new major areas of machine learning applications with high dimensionality and nonconvex objective functions have required the development of novel step-size choosing procedures to reduce the method's overall computing cost (see \cite{Cevher}, \cite{Lan}). The exact or approximate one-dimensional minimization line-search incurs significant computational costs per iteration, particularly when calculating the function value is nearly identical to calculating its derivative and requires the solution of complex auxiliary problems (see \cite{Boyd}). To avoid the line-search, the step-size value may be calculated using prior information such as Lipschitz constants for the gradient. However, this requires using just part of their inexact estimations, which leads to slowdown convergence. This is also true for the well-known divergent series rule (see \cite{Kiwiel}, \cite{Nesterov}).

Although Kiwiel developed the gradient approach for quasiconvex programming in 2001 \cite{Kiwiel}, it results in slow convergence due to the use of diminishing step-size. Following that, there are some improvements to the gradient method, such as the work of Yu et al. (\cite{Yu}, 2019), Hu et al. (\cite{Hu}, 2020), which use a constant step-size but the objective function must fulfil the Hölder condition. The other method is the neurodynamic approach, which uses recurrent neural network models to solve pseudoconvex programming problems with unbounded constraint sets (see \cite{Bian} (Bian et al, 2018), \cite{Liu} (Liu et al., 2022)). Its step-size selection is fixed and not adaptive.

The adaptive step-size procedure was proposed in \cite{Konnov} (Konnov, 2018) and \cite{Ferreira} (Ferreira et. al., 2020) . The method given in \cite{Konnov}, whereby a step-size algorithm for the conditional gradient method is proposed, is effective for solving pseudoconvex programming problems with the boundedness of the feasible set. It has been expanded in \cite{Ferreira} to the particular case of an unbounded feasible set, which cannot be applied to the unconstrained case. In this research, we propose a novel and line-search-free adaptive step-size algorithm for a broad class of programming problems where the objective function is nonconvex and smooth,  the constraint set is unbounded, closed and convex. A crucial component of this procedure is gradually decreasing the step size until a predetermined condition is fulfilled. Although the Lipschitz continuity of the gradient of the objective function is required for convergence, the approach does not make use of predetermined constants. The proposed change has been shown to be effective in preliminary computational tests. We perform various machine learning experiments, including multi-variable logistic regressions and neural networks for classification, to show that the proposed method performs well on large-scale tasks.

The rest of this paper is structured as follows: Section 2 provides preliminary information and details the problem formulation. Section 3 summarizes the primary results, including the proposed algorithms. 
Section 4 depicts numerical experiments and analyzes their computational outcomes. Section 5 presents applications to certain machine learning problems. The final section draws some conclusions.
	\section{Preliminaries}\label{s2}
In the whole paper, we assume that  $C$ is a nonempty, closed and convex set in $ \mathbb{R}^m $, $f:\mathbb{R}^m \to \mathbb R$ is a differentiable function on an open set containing $ C $, the mapping $\nabla f$ is Lipschitz continuous.
	We consider the optimization problem:
	\begin{equation}\label{OP}
		\min\limits_{x\in C} f(x)  \tag{OP($ f,C $)}.
	\end{equation}
 Assume that the solution set of \eqref{OP} is not empty.
 First, we  recall some definitions and basic results that will be used in the next section of the article. The
interested reader is referred to \cite{BauschkeCombettes,Rockafellar} for bibliographical references.  

For $x\in \mathbb{R}^m$, denote by $P_C (x)$ the projection of $x$ onto $C$, i.e., 
$$ 
P_C(x):=\text{argmin}\{\|z-x\|:z\in C\}.
 $$
\begin{proposition}\label{prop_Pc}\textup{\cite{BauschkeCombettes}} It holds that
\begin{itemize}
\item[\textup{(i)}]  $ \| P_C(x)-P_C(y) \| \leq \|x-y\|$ for all $x,y\in \mathbb{R}^m$,
\item[\textup{(ii)}] $\left\langle  y-P_C(x),x-P_C(x)  \right\rangle \leq 0$ for all $x\in \mathbb{R}^m$, $y\in C$.
\end{itemize}
\end{proposition} 
	\begin{definition}\label{def1}\cite{Mangasarian}
		The function $f:\mathbb{R}^m \to \mathbb R$  is said to be 
		\begin{itemize}
			\item[$\bullet$] 		 convex on $ C $ if for all $ x,y\in C $, $ \lambda \in [0,1] $, it holds that
			\begin{equation*}
				f(\lambda x +(1-\lambda)y)\leq \lambda f(x)+(1-\lambda)f(y).
			\end{equation*}
			\item[$\bullet$] 		pseudoconvex on $ C $ if for all $ x,y\in C $, it holds that
			\begin{equation*}
				\langle \nabla f(x), y-x \rangle \geq 0 \Rightarrow f(y)\geq f(x).
			\end{equation*}
			\item[$\bullet$]  quasiconvex on $ C $ if for all $ x,y\in C $, $ \lambda \in [0;1] $, it holds that
			\begin{equation*}
				f(\lambda x +(1-\lambda)y)\leq \max \left\{f(x);f(y) \right\}.
			\end{equation*}
		\end{itemize}
	\end{definition}
	\begin{proposition}\label{p1}\cite{Dennis}
		The differentiable function $ f $ is quasiconvex on $ C $ if and only if 
		\[ 
		f(y)\leq f(x) \Rightarrow 		\langle \nabla f(x), y-x \rangle \leq 0.
		\] 
	\end{proposition}
	It is worth mentioning that "$f$ is convex" $\Rightarrow$ "$f$ is  pseudoconvex" $\Rightarrow$ "$f$ is  quasiconvex"\cite{Mangasarian}.
	\begin{proposition}\label{p1}\cite{Dennis}
		Suppose that $ \nabla f $ is   $L$-Lipschitz continuous on $ C $. For all $ x,y\in C $, it holds that 
		\[ 
		\left|f(y)-f(x) - \left<\nabla f(x),y-x    \right>    \right| \leq \dfrac L 2 \| y-x \|^2.
		\] 
	\end{proposition}	
	\begin{lemma}\label{l1}\cite{Xu}
		Let $ \{a_k \} $; $ \{b_k\} \subset (0;\infty)$ be sequences such that 
		\[ 
		a_{k+1} \leq a_k +b_k\ \forall k\geq 0;\quad \sum_{k=0}^\infty b_k <\infty.
		\]
		Then, there exists the limit $ \lim_{k\to \infty}  a_k=c \in \mathbb R.$
	\end{lemma}
	\section{Main results}\label{s3}
	\begin{algorithm}\label{algGDA}(Algorithm GDA)\\
		\medskip
		\textbf{Step 0.} Choose $x^0 \in C$,  $\lambda_0 \subset (0,+\infty)$, $ \sigma, \kappa \in (0,1) $. Set $k=0$.\\
		\textbf{Step 1.} Given $x^k$ and $ \lambda_k $, compute $x^{k+1}$ and $ \lambda_{k+1} $ as
		\begin{align*}
			& x^{k+1}=P_C (x^k-\lambda_k \nabla f (x^k)) ,\\
			& \textbf{If }  f(x^{k+1}) \leq f(x^k) -\sigma \left< \nabla f (x^k), x^k-x^{k+1}   \right>\textbf{ then set } \lambda_{k+1}:=\lambda_k \textbf{ else set } \lambda_{k+1}:=\kappa \lambda_k.
		\end{align*}
		\textbf{Step 2.}  Update $k :=k+1$. \textbf{If} $ x^{k+1}=x^{k}$ \textbf{ then } STOP \textbf{ else } go to \textbf{Step  1}.
	\end{algorithm}
	\begin{remark}\label{rmk1}
		If Algorithm \ref{algGDA} stops at step $ k $, then $ x^k $ is a stationary point of the problem \ref{OP}.
		Indeed, since $ x^{k+1}=P_C\left( x^{k}-\lambda_k \nabla f(x^k)  \right)$, applying Proposition \ref{prop_Pc}-(ii), we have
		\begin{equation}\label{e2}
			\left < z-x^{k+1},x^k-\lambda_k \nabla f (x^k) -x^{k+1}  \right> \leq 0\ \forall z\in C.
		\end{equation}
		If $ x^{k+1}=x^k $, we get
		\begin{equation}\label{e21}
			\left\langle \nabla f (x^k), z-x^k \right\rangle \geq 0\quad \forall z \in C,
		\end{equation}
		which means $ x^k $ is a stationary point of the problem. If, in addition, $ f $ is pseudoconvex, from \eqref{e21}, it implies that $ f(z)\geq f(x^k) $ for all $ z\in C $, or $ x^k $ is a solution of  \ref{OP}. 
	\end{remark}
	Now, suppose that the algorithms generates an infinite sequence. We will prove that this sequence converges to a solution of the problem \ref{OP}.
	\begin{theorem}\label{theorem1}
		Assume that the sequence $ \left\{ x^k\right\} $ is generated by Algorithm \ref{algGDA}. Then,
		the sequence $ \left\{f(x^k) \right\} $ is convergent and each limit point (if any) of the sequence $ \left\{ x^k\right\} $ is a stationary point of the problem. Moreover,
		\begin{itemize}
			\item[$ \bullet $] if $ f $ is quasiconvex on $ C $, then the sequence $ \left\{x^k \right\} $ converges to a stationary point of the problem.
			\item[$ \bullet $]  if $ f $ is pseudoconvex on $ C $, then the sequence $ \left\{x^k \right\} $ converges to a solution of the problem.
		\end{itemize}
	\end{theorem}
	
	\begin{proof}
		Applying Proposition \ref{p1}, we get
		\begin{equation}\label{e1}
			f(x^{k+1}) \leq f(x^k) + \left< \nabla f(x^k), x^{k+1}-x^k  \right> +\frac L 2 \| x^{k+1}-x^k \| ^2.
		\end{equation}
		In \eqref{e2}, taking $ z=x^k \in C$, we arrive at
		\begin{equation}\label{e3}
			\left< \nabla f(x^k),x^{k+1}-x^k\right> \leq -\dfrac{1}{\lambda_k}\| x^{k+1}-x^k\|^2.
		\end{equation}
		Combining \eqref{e1} and \eqref{e3}, we obtain
		\begin{equation}\label{e4}
			f(x^{k+1}) \leq f(x^k) -\sigma \left< \nabla f(x^k), x^{k}-x^{k+1}  \right> -\left(\frac {1-\sigma} {\lambda_k}- \frac L 2\right) \| x^{k+1}-x^k \| ^2.	
		\end{equation}
		We now claim that $ \{\lambda_k\} $ is bounded away from zero, or in other words, the step size changes finite times. Indeed, suppose, by contrary, that $ \lambda_k \to 0. $ From \eqref{e4}, there exists $ k_0 \in \mathbb N $ satisfying
		\[ 
		f(x^{k+1}) \leq f(x^k) -\sigma \left< \nabla f(x^k), x^{k}-x^{k+1}  \right>\ \forall k\geq k_0.
		\]
		According to the construction of $\lambda_k$, the last inequality implies that $ \lambda_k =\lambda_{k_0} $ for all $ k\geq k_0 $. This is a contradiction. 
		And so, there exists $ k_1 \in \mathbb N $ such that for all $ k\geq k_1 $, we have $ \lambda_k=\lambda_{k_1}$ and
		\begin{equation}\label{e5}
			f(x^{k+1}) \leq f(x^k) -\sigma \left< \nabla f(x^k), x^{k}-x^{k+1}  \right>.
		\end{equation} 
		Noting that $  \left< \nabla f(x^k), x^{k}-x^{k+1}  \right> \geq 0 $, we infer that  the sequence $ \{ f(x^k)\} $ is convergent and 
		\begin{equation}\label{e6}
			\sum_{k=0}^\infty \left< \nabla f(x^k), x^{k}-x^{k+1}  \right> <\infty;\quad 	 \sum_{k=0}^\infty       \|  x^{k+1}-x^k \|^2 <\infty.  
		\end{equation}
		From \eqref{e2}, for all $ z\in C $, we have
		\begin{align}\label{e7}
			\| x^{k+1}-z \|^2 &= 	\| x^{k}-z \|^2  -	\| x^{k+1}-x^k \|^2  +2 \left \langle x^{k+1}-x^k,x^{k+1}-z \right \rangle\notag\\
			&\leq 	\| x^{k}-z \|^2  -	\| x^{k+1}-x^k \|^2  +2 \lambda_k    \left \langle \nabla f(x^k), z-x^{k+1}   \right \rangle.
		\end{align}	
		Let $ \bar x $ be a limit point of $ \left\{ x^k\right\} $. There exists a subsequence $ \left\{x^{k_i} \right\}\subset \left\{x^k \right\} $	 such that $\lim_{i\to\infty} x^{k_i} =\bar x.$ In \eqref{e7}, let $ k=k_i $ and take the limit as $ i\to \infty$. Noting that $ \| x^k-x^{k+1} \|\to 0 $, $ \nabla f $ is continuous, we get
		\[ 
		\left \langle \nabla f(\bar x), z-\bar x   \right \rangle \geq 0\quad \forall z\in C,
		\]
		which means $ \bar x $ is a stationary point of the problem. Now, suppose that $ f $ is quasiconvex on $ C $. Denote 
		\[ 
		U:=\left\{ x \in C: f(x)\leq f(x^k)\quad \forall k\geq0 \right\}.
		\]
		Note that $ U $ contains the solution set of \ref{OP}, and hence, is not empty.
		Take $ \hat x  \in U $.  Since $ f(x^k) \geq f(\hat x)$ for all $ k\geq 0 $, it implies that
		\begin{equation}\label{q1}
			\left \langle \nabla f(x^k),\hat x -x^k \right \rangle \leq 0\quad \forall k\geq 0.   
		\end{equation}
		Combing \eqref{e7} and \eqref{q1}, we get
		\begin{align}\label{e71}
			\| x^{k+1}-\hat x \|^2  &\leq 	\| x^{k}-\hat x \|^2  -	\| x^{k+1}-x^k \|^2  +2 \lambda_k    \left \langle \nabla f(x^k), x^k-x^{k+1}   \right \rangle.
		\end{align}
		Applying Lemma \ref{l1} with $ a_k = 	\| x^{k+1}-\hat x \|^2$, $ b_k =2 \lambda_k    \left \langle \nabla f(x^k), x^k-x^{k+1}   \right \rangle $, we deduce that the sequence  $ \{ \| x^{k}-\hat x \| \} $ is convergent for all $ \hat x \in U $.  Since the sequence $ \{  x^k\} $ is bounded, there exist a subsequence $ \{ x^{k_j}\} \subset \{x^k\}$ such that $ \lim_{i\to \infty} x^{k_i} =\bar x\in C.$ From \eqref{e5}, we know that the sequence $ \left\{ f(x^k) \right\}  $ is nondecreasing and convergent. It implies that $ \lim_{k\to \infty} f(x^k) =f(\bar x) $ and $ f(\bar x) \leq f(x^k) $ for all $ k\geq 0. $ This means $ \bar x \in U $ and  the sequence $ \left\{  \|x^k -\bar x \| \right\}  $ is convergent. Thus,
		\[ 
		\lim\limits_{k\to \infty} \| x^k-\bar x \| =\lim\limits_{i\to \infty} \| x^{k_i}-\bar x \| =0.
		\]
		Note that each limit point of $ \left\{ x^k\right\} $ is a stationary point of the problem. Then, the whole sequence $ \left\{x^k \right\} $ converges to $ \bar x $ - a stationary point of the problem. Moreover, when $ f $ is pseudoconvex, this stationary point becomes a solution of \ref{OP}.  		 \hfill	 $\Box$
	\end{proof}

\begin{remark}\label{rmk2} In Algorithm GDA, we can choose $\lambda_0 = \lambda,$ with the constant number $ \lambda \leq 2(1-\sigma)/L.$ Then we get $(1-\sigma)/\lambda_0 - L/2 \geq 0.$ Combined with \eqref{e4}, it implies that the condition $f(x^{k+1}) \leq f(x^k) -\sigma \left< \nabla f (x^k), x^k-x^{k+1}   \right>$ is satisfied and the step-size $\lambda_k = \lambda$ for all step $k.$ Therefore, Algorithm GDA is still applicable for the constant step-size $ \lambda \leq 2(1-\sigma)/L.$ For any $\lambda \in (0,2/L),$ there exists $\sigma \in (0,1)$ such that $\lambda \leq 2(1-\sigma)/L.$ As a result, if the value of Lipschitz constant $L$ is prior known, we can choose the constant step-size $ \lambda \in (0,2/L)$ as in the gradient descent algorithm for solving convex programming problems. The gradient descent method with constant step-size for quasiconvex programming problems, a particular version of Algorithm GDA, is shown below.
\end{remark}

\begin{algorithm}\label{algGDA}(Algorithm GD)\\
		\medskip
		\textbf{Step 0.} Choose $x^0 \in C$,  $\lambda \subset (0,2/L)$. Set $k=0$.\\
		\textbf{Step 1.} Given $x^k,$ compute $x^{k+1}$ as
		$$x^{k+1}=P_C (x^k-\lambda \nabla f (x^k))$$ 
		\textbf{Step 2.}  Update $k :=k+1$. \textbf{If} $ x^{k+1}=x^{k}$ \textbf{ then } STOP \textbf{ else } go to \textbf{Step  1}.
	\end{algorithm}

Next, we estimate the convergence rate of Algorithm \ref{algGDA}  in solving  unconstrained optimization problems.
\begin{corollary}
	Assume that $f$ is convex, $ C=\mathbb R^m $ and $\{x^k\}$ is the sequence generated by Algorithm \ref{algGDA}. Then,
	$$ 
f(x^k)-f(x^*) =O \left( \frac 1 k\right) , 
 $$
where $x^*$ is a solution of the problem.
	\end{corollary}
\begin{proof}
Let $ x^* $ be a solution of the problem. Denote $ \Delta_k:= f(x^k) -f(x^*) $.	
From \eqref{e5}, noting that $ x^k-x^{k+1}=\lambda_k \nabla f (x^k) $, we have
 \begin{equation}\label{f1}
\Delta_{k+1} \leq \Delta_{k} -\sigma \lambda_{k_1} \| \nabla f(x^k)  \|^2   \quad \forall k\geq k_1.
 \end{equation}
 On the other hand, since the sequence $ \left\{x^k \right\} $ is bounded and $ f $ is convex, it holds that
 \begin{align}\label{f2}
 0\leq \Delta_{k}  & \leq \left\langle  \nabla f(x^k), x^k-x^*   \right\rangle \notag \\
 & \leq M \|  \nabla f(x^k) \|,
\end{align}
where $ M:=\sup\left\{ \| x^k-x^* \|:k\geq k_1  \right\} <\infty$. From \eqref{f1} and \eqref{f2}, we arrive at
\begin{equation}\label{f3}
	\Delta_{k+1} \leq \Delta_{k} - Q  \Delta_{k}^2  \quad \forall k\geq k_1,
\end{equation}
where $ Q:= \frac {\sigma \lambda_{k_1} }{M^2 }$. Noting that $ \Delta_{k+1} \leq \Delta_{k} $, from \eqref{f3}, we obtain\\
 \begin{equation*}
\dfrac 1 {\Delta_{k+1}} \geq \dfrac 1 {\Delta_{k}}+Q\geq \ldots \geq  \dfrac 1 {\Delta_{k_1}}+(k-k_1)Q,
 \end{equation*}
which implies
\[ 
 f(x^k) -f(x^*)= O \left( \frac 1 k \right).
 \]
	\end{proof}

We investigate the stochastic variation of Algorithm GDA for application in large scale deep learning. The problem is stated as follows:
	$$
	\min _x \mathbb{E}\left[f_{\xi}(x)\right],
	$$
	where $\xi$ is the stochastic parameter and the function $f_{\xi}$ is $L$-smooth. We are generating a stochastic gradient $\nabla f_{\xi^k}\left(x^k\right)$ by sampling $xi$ at each iteration $k.$ The following Algorithm SGDA contains a detailed description.
	\begin{algorithm}\label{algSGDA}(Algorithm SGDA)\\
		\medskip
		\textbf{Step 0.} Choose $x^0 \in C$,  $\lambda_0 \subset (0,+\infty)$, $ \sigma, \kappa \in (0,1) $. Set $k=0$.\\
		\textbf{Step 1.} Sample $\xi^k$ and compute $x^{k+1}$ and $ \lambda_{k+1} $ as
		\begin{align*}
			& x^{k+1}=P_C (x^k-\lambda_k \nabla f_{\xi^k} (x^k)) ,\\
			& \textbf{If }  f_{\xi^k}(x^{k+1}) \leq f_{\xi^k} (x^k) -\sigma \left< \nabla f_{\xi^k} (x^k), x^k-x^{k+1}   \right>
			\textbf{ then set } \lambda_{k+1}:=\lambda_k \textbf{ else set } \lambda_{k+1}:=\kappa \lambda_k.
		\end{align*}
		\textbf{Step 2.}  Update $k :=k+1$. \textbf{If} $ x^{k+1}=x^{k}$ \textbf{ then } STOP \textbf{ else } go to \textbf{Step  1}.
	\end{algorithm}
	
	\section{Numerical experiments}
We use two existing cases and two large-scale examples with variable sizes to validate and demonstrate the performance of the proposed method. All tests are carried out in Python on a MacBook Pro M1 ($3.2 \mathrm{GHz}$ 8-core processor, $8.00 \mathrm{~GB}$ of RAM). The stopping criterion in the following cases is ``the number of iterations $\leq$ \#Iter" where \#Iter is the maximum number of iterations. Denote $x^*$ as the limit point of iterated series $\{x^k\}$ and Time the CPU time of the GDA algorithm using the stop criterion.

Nonconvex objective and constraint functions are illustrated in Examples $1$ and $2$ taken from \cite{Liu}. In addition, Example $2$ is more complex than Example $1$ regarding the form of functions. Implementing Example $3$ with a convex objective function and a variable dimension $n$ allows us to evaluate the proposed method compared to Algorithm GD. Example $3$ is implemented with a pseudoconvex objective function and several values for $n$ so that we may estimate our approach compared to Algorithm RNN.

To compare to neurodynamic method in \cite{Liu}, we consider Problem \ref{OP} with the constraint set is determined specifically by $C = \{x \in \mathbb{R}^n \mid {g}(x) \leq 0,\; Ax = b\},$
where $g(x):=\left(g_1(x), g_2(x), \ldots, g_m(x)\right)^{\top}$ and $g_i: \mathbb{R}^n \rightarrow \mathbb{R}, i=1, \ldots, m$  are differential quasiconvex, the matrix $A \in \mathbb{R}^{p \times n}$ and $b=\left(b_1, b_2, \ldots, b_p\right)^{\top} \in \mathbb{R}^p$. Recall that, in \cite{Liu}, the authors introduced the function
$$
\Psi(s)=\left\{\begin{array}{l}
1, \quad s>0 ; \\
{[0,1], \quad s=0 ;} \\
0, \quad s<0 .
\end{array}\right\},
$$
and
$$
P(x)=\sum_{i=1}^m \max \left\{0, g_i(x)\right\}.
$$
The neorodynamic algorithm established in \cite{Liu}, which uses recurrent neural network (RNN) models for solving Problem \ref{OP} in the form of differential inclusion as follows:
\begin{equation}
    \label{model} \tag{RNN}
    \dfrac{d}{d t} x(t) \in-c(x(t)) \nabla f(x(t))-\partial P(x(t))-\partial\|A x(t)-b\|_1
\end{equation}
where the adjusted term $c(x(t))$ is
$$
c(x(t))=\left\{\prod_{i=1}^{m+p} c_i(t) \mid c_i(t) \in 1-\Psi\left(J_i(x(t))\right), i=1,2, \ldots, m+p\right\}
$$
with
$$
J(x)=\left(g_1(x), \ldots, g_m(x),\left|A_1 x-b_1\right|, \ldots,\left|A_p x-b_p\right|\right)^{\top},
$$
the subgradient term of $P(x)$ is
$$
\partial P(x)= \begin{cases}0, & x \in \operatorname{int}\left(X\right) \\ \sum_{i \in I_0(x)}[0,1] \nabla g_i(x), & x \in \operatorname{bd}\left(X\right) \\ \sum_{i \in I_{+}(x)} \nabla g_i(x)+\sum_{i \in I_0(x)}[0,1] \nabla g_i(x), & x \notin X,\end{cases}
$$
with
$$ X = \{ x: g_i(x) \leq 0, i=1,2 \dots m\},$$
$$
I_{+}(x)=\left\{i \in\{1,2, \ldots, m\}: g_i(x)>0\right\},
$$
 $$I_0(x)=\left\{i \in\{1,2, \ldots, m\}: g_i(x)=0\right\},$$
and the subgradient term of $\|A x-b\|_1$ is
$$
\partial\|A x-b\|_1=\sum_{i=1}^p\left(2 \Psi\left(A_i x-b_i\right)-1\right) A_i^{\top},
$$
with $A_i \in \mathbb{R}^{1 \times n}(i=1,2, \ldots, p)$ be the row vectors of the matrix $A.$
	
	\begin{example} First, let's look at a simple nonconvex problem \ref{OP}:
		$$\begin{array}{ll}\text { minimize } & f(x)=\dfrac{x_{1}^{2}+x_{2}^{2}+3}{1+2 x_{1}+8 x_{2}} \\ \text { subject to } & x\in C,\end{array}$$ where $C = \{x=\left(x_{1}, x_{2}\right)^{\top} \in \mathbb{R}^{2}| g_{1}(x)=-x_{1}^{2}-2 x_{1} x_{2} \leq-4;\; x_{1}, x_{2} \geq 0\}.$ It is quite evident that for this problem, the objective function $f$ is pseudoconvex on the convex feasible set (Example 5.2 \cite{Liu}, Liu et al., 2022).
		
Fig. \ref{fig:ex1} illustrates the temporary solutions of the proposed method for various initial solutions. It demonstrates that the outcomes converge to the optimal solution $x^*=(0.8922,1.7957)$ of the given problem. The objective function value generated by Algorithm GDA is $0.4094$, which is better than the optimum value $0.4101$ of the neural network in Liu et al. (2022) \cite{Liu}.

\begin{figure}[h]
\centering
\includegraphics[scale = 0.25]{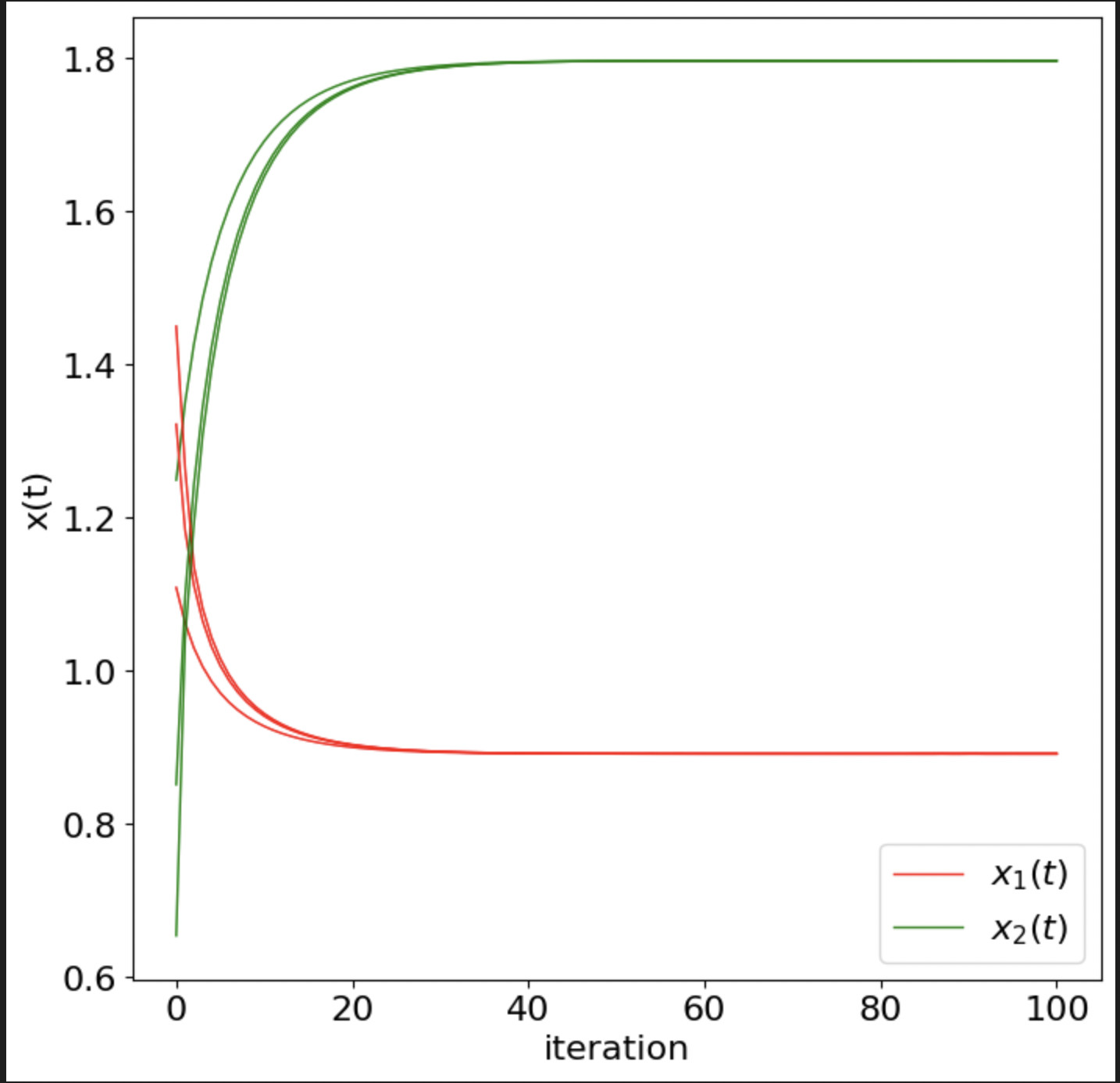}
\caption{Computational results for Example 1.}
\label{fig:ex1}
\end{figure}

	\end{example}
	
	\begin{example} Consider the following nonsmooth pseudoconvex optimization problem with nonconvex inequality constraints (Example 5.1 \cite{Liu}, Liu et al., 2022):
		$$
		\begin{array}{ll}
			\text { minimize } & f(x)=\dfrac{e^{\left|x_{2}-3\right|}-30}{x_{1}^{2}+x_{3}^{2}+2 x_{4}^{2}+4} \\
			\text { subject to } & g_{1}(x)=\left(x_{1}+x_{3}\right)^{3}+2 x_{4}^{2} \leq 10, \\
			& g_{2}(x)=\left(x_{2}-1\right)^{2} \leq 1, \\
			& 2 x_{1}+4 x_{2}+x_{3}=-1,
		\end{array}
		$$
		where $x=\left(x_{1}, x_{2}, x_{3}, x_{4}\right)^{\top} \in \mathbb{R}^{4}$. The objective function $f(x)$ is nonsmooth pseudoconvex on the feasible region $C,$ and the inequality constraint $g_{1}$ is continuous and quasiconvex on $C$, but not pseudoconvex (Example 5.1 \cite{Liu}). It is easily verified that $x_2 \not=3$ for any $x\in C.$ Therefore, the gradient vector of $\left|x_{2}-3\right|$ is $(x_{2}-3)/\left|x_{2}-3\right|$ for any $x\in C.$ From that, we can establish the gradient vector of $f(x)$ at a point $x \in C$ used in the algorithm.  	
		Fig \ref{fig:ex2} shows that Algorithm GDA converges to an optimal solution  $x^{*}=(-1.0649,0.4160,-0.5343,0.0002)^{\top}$ with the optimal value  $-3.0908$, which is better than the optimal value $-3.0849$ of the neural network model in \cite{Liu}.
		
		\begin{figure}[h]
    \centering
    \includegraphics[scale = 0.25]{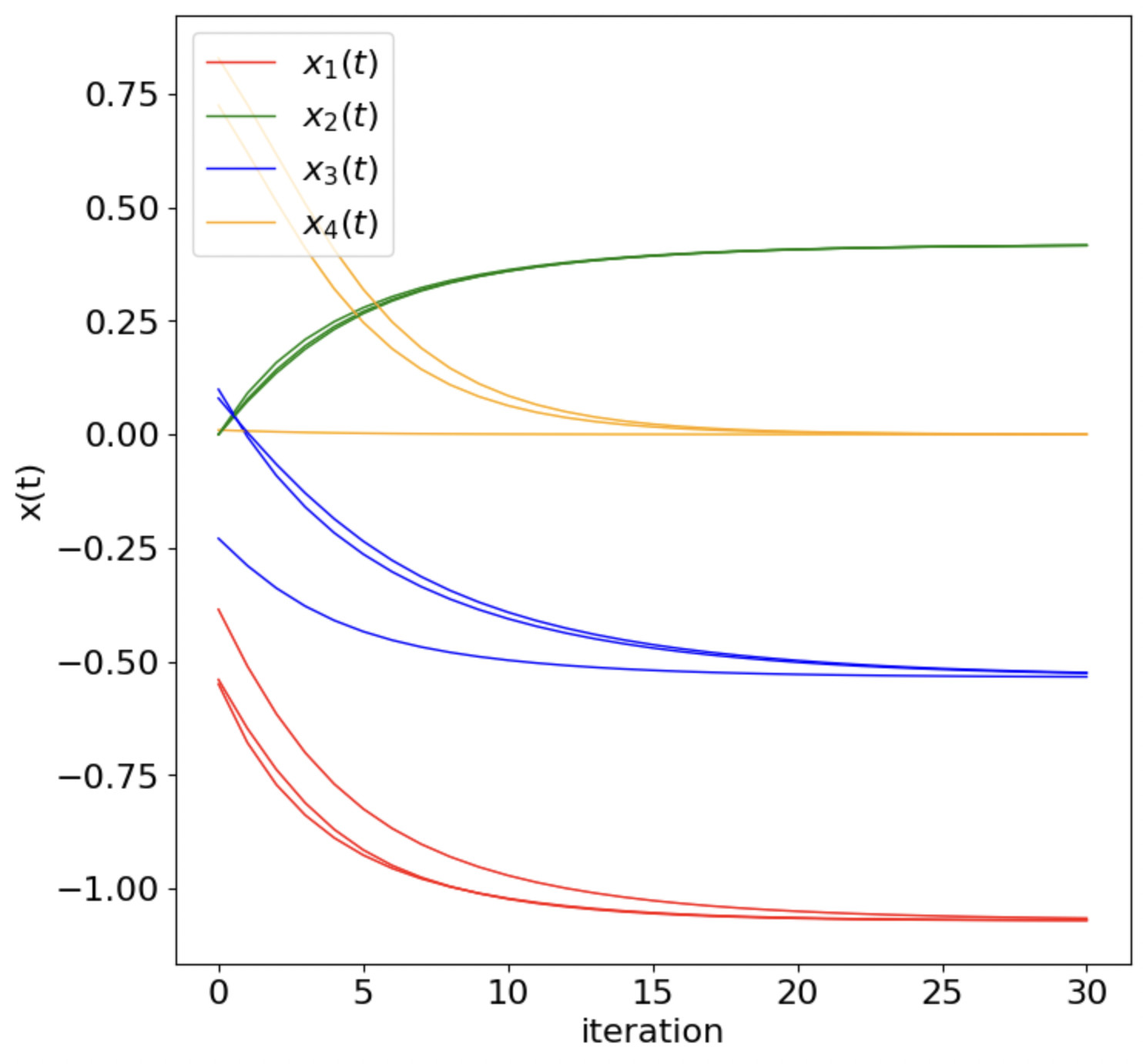}
    \caption{Computational results for Example 2.}
    \label{fig:ex2}
    \end{figure}
	\end{example}

	\begin{example}  Let $e:=(1, \ldots, n) \in \mathbb{R}^{n}$ be a vector, $\alpha>0$ and $\beta>0$ be constants satisfying the parameter condition $2 \alpha>3 \beta^{3 / 2} \sqrt{n}$. Consider Problem \ref{OP} (Example 4.5 \cite{Ferreira}, Ferreira et. al, 2022) with the associated function
		$$
		f(x):=a^{T} x+\alpha x^{T} x+\frac{\beta}{\sqrt{1+\beta x^{T} x}} e^{T} x,
		$$
		with $a \in \mathbb{R}_{++}^{n}$ is convex and the nonconvex  constraint is given by $$C:=\left\{x \in \mathbb{R}_{++}^{n}: 1 \leq x_{1} \ldots x_{n}\right\}.$$ This example is implemented to compare Algorithm GDA with the original gradient descent algorithm (GD). We choose a random number $\beta = 0.741271, \alpha = 3\beta^{3/2}\sqrt{n+1}$ fulfilled the parameter condition and Lipschitz coefficient $L = \left(4 \beta^{3 / 2} \sqrt{n}+3 \alpha\right)$ suggested in \cite{Ferreira}. The step size of Algorithm GD is $\lambda = 1/L$ and the initial step size of Algorithm GDA is ${\lambda}_0 = 5/L.$ Table \ref{table_Ex3} shows the optimal value, number of loops, computational time of two algorithms through the different dimensions. From this result, Algorithm GDA is more efficient than GD at both the computational time and the optimal output value, specially for the large scale dimensions.
		\begin{table}[h!]
\centering
\renewcommand{\arraystretch}{2}
\begin{tabular}{rrrrrrr}
\hline \multirow{2}{*}{$\quad n$} & \multicolumn{3}{c}{ Algorithm GDA (proposed)} & \multicolumn{3}{c}{ Algorithm GD } \\ \cline{2-7}
 & $f\left(x^*\right)$ & \#Iter & Time & $f\left(x^*\right)$ & \#Iter & Time \\
\hline \hline {$\quad$10} & $\textbf{79.3264}$ & 9 & $\textbf{0.9576}$ & $79.3264$ & 15 & $1.5463$  \\
\hline \multirow{1}{*}{$\quad$20} & $\textbf{220.5622}$ & 10 & $\textbf{6.0961}$ & $220.5622$ & 67 & $34.0349$ \\
\hline \multirow{1}{*}{$\quad$50} & $\textbf{857.1166}$ & 12 & $\textbf{2.8783}$ & $857.1166$ & 16 & $4.6824$ \\
\hline \multirow{1}{*}{$\quad$100} & $\textbf{2392.5706}$ & 12 & $\textbf{17.2367}$ & $2392.5706$ & 17 & $30.8886$ \\
\hline \multirow{1}{*}{$\quad$200} & $\textbf{7065.9134}$ & 65 & $\textbf{525.1199}$ & $7179.3542$ & 200 & $1610.6560$ \\
\hline \multirow{1}{*}{$\quad$500} & $\textbf{26877.7067}$ & 75 & $\textbf{2273.0011}$ & $27145.6292$ & 500 & $14113.5003$ \\
\hline
\end{tabular}
	    \caption{Computational results for Example 3.}
	    \label{table_Ex3}
\end{table}
		
\end{example} 

\begin{example} 
To compare Algorithm GDA to Algorithm RNN (Liu et al., 2022 \cite{Liu}), consider Problem \ref{OP} with the objective function $$f(x)=-\exp \left(-\sum_{i=1}^{n} \frac{x_i^2}{\varrho_i^2}\right)$$ which is pseudoconvex on the convex constraint set
$$ C:= \{Ax=b, g(x) \leq 0\},$$
where $x \in \mathbb{R}^{n},$ the parameter vector $\varrho=\left(\varrho_1, \varrho_2, \ldots, \varrho_{n}\right)^{\top}$ with $\varrho_i>0,$ the matrix $A = (a_1, a_2,\ldots,a_n) \in \mathbb{R}^{1 \times n}$ with $a_{i} = 1$ for $1 \leq i\leq n/2$ and $a_{i} = 3$ for $n/2 <i \leq n,$ and the number $ b=16$. The inequality constraints are $$g_i(x)=x_{10 \cdot(i-1)+1}^2+x_{10 \cdot(i-1)+2}^2+\cdots+x_{10 \cdot(i-1)+10}^2-20,$$ for $i=1,2,\ldots, n/10.$ Table \ref{table_Ex4} presents the computational results of Algorithms GDA and RNN. Note that the function $-\ln (-z)$ is monotonically increasing by $z \in \mathbb{R}, z<0.$ Therefore, to compare the approximated optimal value through iterations, we compute $-\ln (-f(x^*))$ instead of $f(x^*),$ with approximated optimal solution $x^*.$ For each $n,$ we solve Problem \ref{OP} to find the value $-\ln (-f(x^*)),$ the number of iterations (\#Iter) to terminate the algorithms, and the computation time (Time) by seconds. The computational results reveal that the proposed algorithm outperforms Algorithm RNN in the test scenarios for both optimum value and computational time, specially when the dimensions is large scale. 

	\begin{table}[h!]
\centering
\renewcommand{\arraystretch}{2}
\begin{tabular}{rrp{0.05\textwidth}rrp{0.05\textwidth}r}
\hline \multirow{2}{*}{$\quad n$} & \multicolumn{3}{c}{ Algorithm GDA (proposed) } & \multicolumn{3}{c}{ Algorithm RNN } \\ \cline{2-7}
 & $-\ln (-f(x^*))$ & \#Iter & Time & $-\ln (-f(x^*))$ & \#Iter & Time \\
\hline \hline 10 & $\textbf{5.1200}$ & 10 & \textbf{0.3} & $5.1506$ & 1000 & $256$ \\
\hline 20 & $\textbf{2.5600}$ & 10 & \textbf{0.8} & $2.5673$ & 1000 & $503$ \\
\hline 50 & $\textbf{1.0240}$ & 10 & \textbf{2} & $1.0299$ & 1000 & $832$ \\
\hline 100 & $\textbf{0.5125}$ & 10 &\textbf{7} & $13.7067$ & 1000 & $1420$ \\
\hline 300 & $\textbf{15.7154}$ & 10 & \textbf{84} & $39.3080$ & 1000 & $3292$ \\
\hline 400 & $\textbf{20.9834}$ & 10 & \textbf{163} & $57.6837$ & 1000 & $4426$ \\
\hline 600 & $\textbf{29.0228}$ & 10 & \textbf{371} & $83.6265$ & 1000 & $6788$ \\
\hline
\end{tabular}
	    \caption{Computational results for Example 4.}
	    \label{table_Ex4}
\end{table}

\end{example}

\section{Applications to machine learning}
The proposed method, like the GD algorithm, has many applications in machine learning. We analyze three common machine learning applications, namely supervised feature selection, regression, and classification, to demonstrate the accuracy and computational efficiency compared to other alternatives.

Firstly, the feature selection problem can be modeled as a minimization problem of a pseudoconvex fractional function on a convex set, which is a subclass of Problem \ref{OP}. This problem is used to compare the proposed approach to neurodynamic approaches. Secondly, since the multi-variable logistic regression problem is a convex programming problem, the GDA algorithm and available variants of the GD algorithm can be used to solve it. Lastly, a neural network model for an image classification problem is the same as a programming problem with neither a convex nor a quasi-convex objective function. For training this model, we use the stochastic variation of the GDA method (Algorithm SGDA) as a heuristic technique. Although the algorithm's convergence cannot be guaranteed like the cases of pseudoconvex and quasiconvex objective functions, the theoretical study showed that if the sequence of points has a limit point, it converges into a stationary point of the problem (see Theorem \ref{theorem1}). Computational experiments indicate that the proposed method outperforms existing neurodynamic and gradient descent methods.
\subsection{Supervised feature selection}

Feature selection problem is carried out on the dataset with $p$-feature set $\mathcal{F}=\left\{F_{1}, \ldots, F_{p}\right\}$  and $n$-sample set $\left\{\left(x_{i}, y_{i}\right) \mid i=\right.$ $1, \ldots, n\}$, where $x_{i}=\left(x_{i 1}, \ldots, x_{i p}\right)^{\mathrm{T}}$ is a $p$-dimensional feature vector of the $i$th sample and $y_{i} \in\{1, \ldots, m\}$ represents the corresponding labels indicating classes or target values. In \cite{Wang}, an optimal subset of $k$ features $\left\{F_{1}, \ldots, F_{k}\right\} \subseteq \mathcal{F}$ is chosen with the least redundancy and the highest relevancy to the target class $y$. The feature redundancy is characterized by a positive semi-definite matrix $Q.$ Then the first aim is minimizing the convex quadratic function $w^{T} Q w.$ The feature relevance is measured by $\rho^{T} w,$ where   $\rho=$ $\left(\rho_{1}, \ldots, \rho_{p}\right)^{T}$ is a  relevancy parameter vector. Therefore, the second aim is maximizing the linear function $\rho^{T} w.$ Combining two goals deduces the equivalent problem as follows:  
	 \begin{equation}\label{fs_problem}
	    \begin{array}{ll}\text { minimize } & \dfrac{w^{T} Q w}{\rho^{T} w} \\ \text { subject to } & e^{T} w=1 \\ & w \geq 0,\end{array}
	\end{equation}
where $w=\left(w_{1}, \ldots, w_{p}\right)^{T}$ is the feature score vector to be determined. Since the objective function of problem \eqref{fs_problem} is the fraction of a convex function over a positive linear function, it is pseudoconvex on the constraint set. Therefore, we can solve  \eqref{fs_problem} by Algorithm GDA. 

In the experiment, we implement the algorithms with the Parkinsons dataset, which has 23 features and 197 samples, downloaded at https://archive.ics.uci.edu/ml/datasets/parkinsons. The similarity coefficient matrix $Q$ is determined by $Q=\delta I_p+S$ (see \cite{Wang}), where $p \times p$-matrix $S = (s_{ij})$ with $$s_{i j}=\max \left\{0, \frac{I\left(F_i ; F_j ; y\right)}{H\left(F_i\right)+H\left(F_j\right)}\right\},$$ the information entropy of a random variable vector $\hat{X}$ is 
$$H(\hat{X})=-\sum_{\hat{x} \in \hat{X}} p(\hat{x}) \log p(\hat{x}),$$ 
the multi-information of three random vectors $\hat{X}, \hat{Y}, \hat{Z}$ is  $I(\hat{X} ; \hat{Y} ; \hat{Z}) =I(\hat{X} ; \hat{Y})-I(\hat{X}, \hat{Y} \mid \hat{Z})$ 
with the mutual information of two random vectors  $\hat{X}, \hat{Y}$ defined by $$I(\hat{X} ; \hat{Y})=\sum_{\hat{x} \in \hat{X}} \sum_{\hat{y} \in \hat{Y}} p(\hat{x}, \hat{y}) \log \frac{p(\hat{x}, \hat{y})}{p(\hat{x}) p(\hat{y})}$$  and the conditional mutual information between $\hat{X}, \hat{Y}$ and $\hat{Z}$ defined by $$I(\hat{X} ; \hat{Y} \mid \hat{Z})=\sum_{\hat{x} \in \hat{X}} \sum_{\hat{y} \in \hat{Y}} \sum_{\hat{z} \in \hat{z}} p(\hat{x}, \hat{y}, \hat{z}) \log \frac{p(\hat{x}, \hat{y} \mid \hat{z})}{p(\hat{x} \mid \hat{z}) p(\hat{y} \mid \hat{z})}.$$
The feature relevancy vector $\rho = (\rho_{1}, \ldots, \rho_{p})^{T}$ is determined by Fisher score
$$\rho\left(F_i\right)=\frac{\sum_{j=1}^K n_j\left(\mu_{i j}-\mu_i\right)^2}{\sum_{j=1}^K n_j \sigma_{i j}^2},$$
where $n_j$ denotes the number of samples in class $j, \mu_{i j}$ denotes the mean value of feature $F_i$ for samples in class $j, \mu_i$ is the mean value of feature $F_i$, and $\sigma_{i j}^2$ denotes the variance value of feature $F_i$ for samples in class $j$.

The approximated optimal value of problem \eqref{fs_problem} is $f(w^*)=0.153711$ with the computing time $T= 6.096260$s for the proposed algorithm, while $f(w^*)=0.154013, T=11.030719$ for Algorithm RNN. In comparison to the Algorithm RNN, our algorithm outperforms it in terms of both accuracy and computational time.

\subsection{Multi-variable logistic regression} 
The experiments are performed with the dataset including $\mathbf{{N}}$ observations $(\mathbf{a}_i,b_i)\in \mathbb{R}^{d} \times \mathbb{R}, i=1, \ldots, n.$ The cross-entropy loss function for multi-variable logistic regression is given by
    $J(x)=-\sum_{i=1}^{\mathbf{N}}\big(b_{i}\log\left(\sigma\left(-{x}^{T}\mathbf{a}_{i}\right)\right)+(1-b_{i})\log\left(1-\sigma\left(-{x}^{T}\mathbf{a}_{i}\right)\right)\big),$ where $\sigma$ is the sigmoid function. Associated with $\ell_{2}$-regularization, we get the regularized loss function  $\bar{J}(x) = J(x) + \frac{1}{2\mathbf{N}}\|x\|^{2}.$ The Lipschitz coefficient $L$ is estimated by $\frac{1}{2\mathbf{N}} (\|A\|^{2}/2+1),$ where $A=\left(a_{1}^{\top}, \ldots, a_{n}^{\top}\right)^{\top}.$ We compare the algorithms for training the logistic regression problem by using datasets Mushrooms and W8a (see in \cite{Malitsky}). The GDA method is compared to the GD algorithm with a step size of $1/L$ and Nesterov's accelerated method. The computational results are shown in Fig. \ref{ex_reg01} and Fig. \ref{ex_reg02} respectively. The figures suggest that Algorithm GDA outperforms Algorithm GD and Nesterov's accelerated method in terms of objective function values during iterations. In particular, Fig. \ref{ex_reg02} shows the change of the objective function values according to different $\kappa$ coefficients. In this figure, the notation ``GDA\_0.75" respects to the case $\kappa = 0.75.$ Fig. \ref{ex_reg03} presents the reduction of step-sizes from an initial step-size with respect to the results in Fig. \ref{ex_reg02}.

	\begin{figure}[h]
\centering
\includegraphics[scale = 0.4]{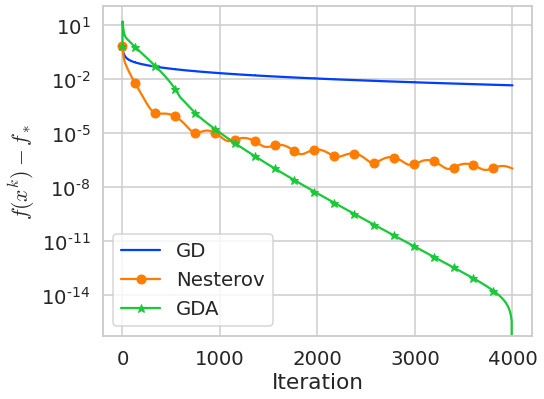}
\caption{The computational results for logistic regression with dataset Mushrooms.}
\label{ex_reg01}
\end{figure}

	\begin{figure}[h]
\centering
\includegraphics[scale = 0.3]{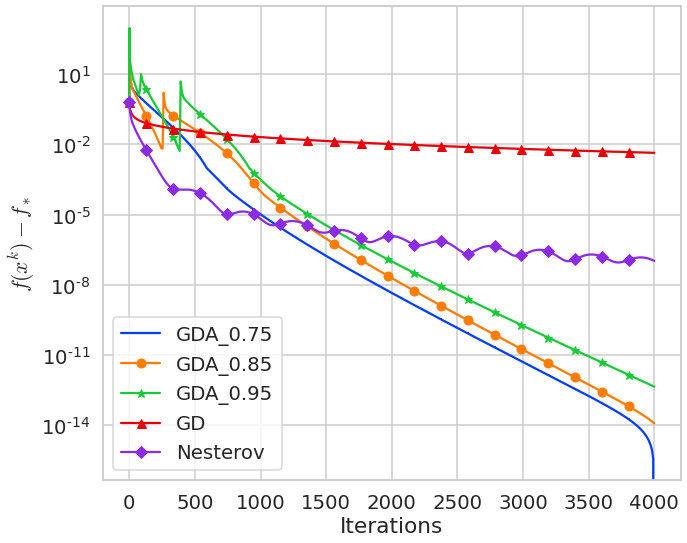}
\caption{The computational results for logistic regression with dataset W8a.}
\label{ex_reg02}
\end{figure}

	\begin{figure}[h]
\centering
\includegraphics[scale = 0.35]{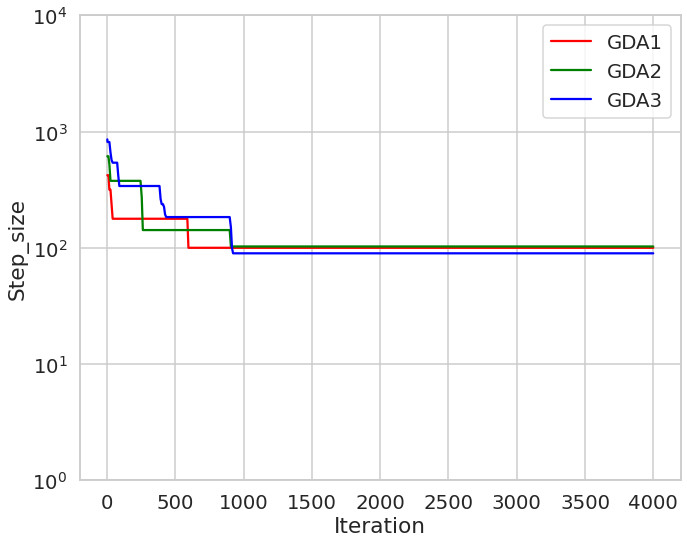}
\caption{The step-sizes changing for each iteration in logistic regression with dataset W8a.}
\label{ex_reg03}
\end{figure}
	
	\subsection{Neural networks for classification}
In order to provide an example of how the proposed algorithm can be implemented into a neural network training model, we will use the standard ResNet-18 architectures that have been implemented in PyTorch and will train them to classify images taken from the Cifar10 dataset downloaded at https://www.cs.toronto.edu/~kriz/cifar.html, while taking into account the cross-entropy loss. In the studies with ResNet-18, we made use of Adam's default settings for its parameters. 

For training this neural network model, we use the stochastic variation of the GDA method (Algorithm SGDA) to compare to Stochastic Gradient Descent (SGD) algorithms.  The computational outcomes shown in Fig. \ref{ex_ANN01} and  Fig. \ref{ex_ANN02} reveal that Algorithm SGDA outperforms Algorithm SGD in terms of testing accuracy and train loss over iterations. 
	\begin{figure}[h]
\centering
\includegraphics[scale = 0.4]{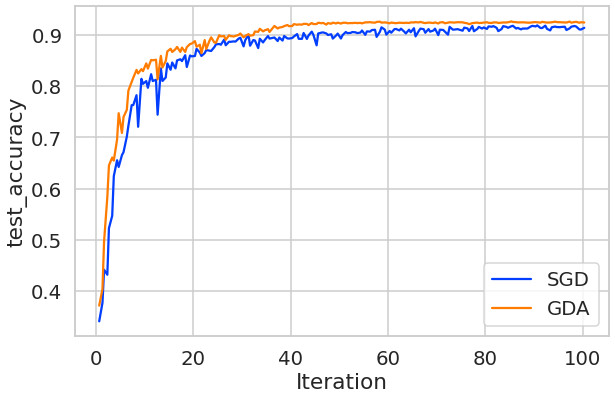}
\caption{The test accuracy through iterations for ResNet-18 model.}
\label{ex_ANN01}
\end{figure}

	\begin{figure}[h]
\centering
\includegraphics[scale = 0.4]{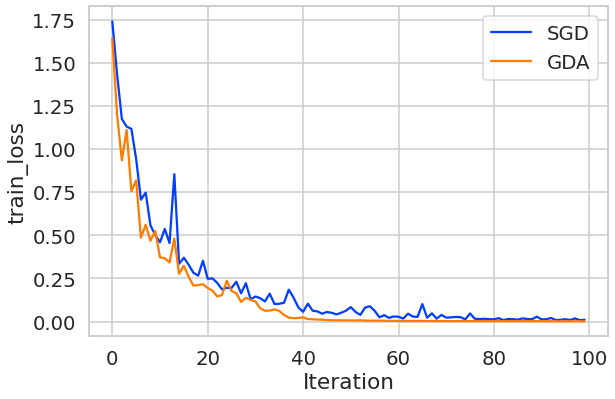}
\caption{The training loss through iterations for ResNet-18 model.}
\label{ex_ANN02}
\end{figure}

\section{Conclusion}
We proposed a novel easy adaptive step-size process in a wide family of solution methods for optimization problems with non-convex objective functions. This approach does not need any line-searching or prior knowledge, but rather takes into consideration the iteration sequence's behavior. As a result, as compared to descending line-search approaches, it significantly reduces the implementation cost of each iteration. We demonstrated technique convergence under simple assumptions. We demonstrated that this new process produces a generic foundation for optimization methods. The preliminary results of computer experiments demonstrated the new procedure's efficacy.
	
\section*{Availability of data and materials}
The data that support the findings of this study are available from the corresponding author,
upon reasonable request.

\section*{Acknowledgments}
This work was supported by Vietnam Ministry of Education and Training.

\end{document}